\documentclass[
final
 , nomarks
]{dmtcs-episciences}

\usepackage[utf8]{inputenc}
\usepackage{subfigure}

\usepackage{amsmath}
\usepackage{amssymb}
\usepackage{amsthm}
\theoremstyle{plain}
\newtheorem{theorem}{Theorem}

\newtheorem{proposition}{Proposition}
\newtheorem{lemma}[theorem]{Lemma}

\theoremstyle{definition}

\newcommand{\sd}{\triangle}
\newcommand{\F}{\mathcal{F}}
\newcommand{\ZZ}{{\mathbb{Z}}}
\newcommand{\X}{\mathcal{X}}
\newcommand{\Y}{\mathcal{Y}}

\author{Kenjiro Takazawa\thanks{This work is partially supported by 
JSPS KAKENHI Grant Numbers 
JP16K16012, 
JP20K11699, 
Japan.}}
\title[Notes on equitable partitions into matching forests and into $b$-branchings]{Notes on equitable partitions into matching forests in mixed graphs and into $b$-branchings in digraphs}
\affiliation{
  Department of Industrial and Systems Engineering, Faculty of Science and Engineering, 
  Hosei University, Japan. }
\received{2021-11-18}

\revised{2022-03-10}

\accepted{2022-03-11}

\begin{document}
\publicationdetails{24}{2022}{1}{13}{8719}
\maketitle
\begin{abstract}
An equitable partition into branchings in a digraph
is a partition of the arc set into branchings such that 
the sizes of any two branchings differ at most by one. 
For a digraph whose arc set can be partitioned into $k$ branchings, 
there always exists an equitable partition into $k$ branchings. 
In this paper, 
we present two extensions of equitable partitions into branchings in digraphs: those into matching forests in mixed graphs; and 
into $b$-branchings in digraphs. 
For matching forests, 
Kir\'{a}ly and Yokoi (2022) considered a tricriteria equitability based on 
the sizes of the matching forest, 
and the matching and branching therein. 
In contrast to this, 
we introduce a single-criterion equitability based on the number of covered vertices, 
which is plausible in the light of the delta-matroid structure of matching forests. 
While the existence of this equitable partition can be derived from a lemma in Kir\'{a}ly and Yokoi, 
we present its direct and simpler proof. 
For $b$-branchings, 
we define 
an 
equitability notion based on the size of the $b$-branching and the 
indegrees of all vertices, 
and 
prove that an equitable partition always exists. 
We then derive the integer decomposition property of the associated polytopes. 
\end{abstract}

\section{Introduction}
\label{SECintro}

Partitioning a finite set into its subsets with certain combinatorial structure is a fundamental topic in 
the fields of 
combinatorial optimization, discrete mathematics, and graph theory. 
The most typical partitioning problem is graph coloring, 
which amounts to partitioning the vertex set of a graph into stable sets. 
In particular, 
\emph{equitable coloring}, 
in which the numbers of vertices in any two stable sets differ at most by one, 
has attracted researchers' interest since the famous conjecture of Erd\H{o}s \cite{Erd64} 
on the existence of an equitable coloring with $\Delta + 1$ colors in a graph with maximum degree $\Delta$, 
which was later proved by Hajnal and Szemer\'{e}di \cite{HS70}. 

Equitable edge-coloring has been mainly considered in bipartite graphs: 
a bipartite graph with maximum degree $\Delta$ 
admits an equitable edge-coloring with $k$ colors for every $k \ge \Delta$  \cite{deW70,deW72,DM69,FF69,McD72}. 
Extension of equitable edge-coloring in bipartite graphs to equitable partition of 
the common ground set of two matroids into common independent sets has been a challenging 
topic in the literature \cite{DM76,FTY20}. 

One successful example is equitable partition into branchings in digraphs. 
Let $(V,A)$ denote a digraph with vertex set $V$ and arc set $A$. 
In a digraph $(V,A)$, an arc subset $B \subseteq A$ is a \emph{branching} if every vertex has at most one incoming arc in $B$ and $B$ 
includes no directed cycle. 
The following theorem is derived from Edmonds' disjoint branchings theorem \cite{Edm73}. 
For a real number $x$, 
let $\lfloor x \rfloor$ and $\lceil x \rceil$ denote 
the maximum integer that is not greater than $x$ 
and 
the minimum integer that is not less than $x$, 
respectively. 
\begin{theorem}[see Schrijver {\cite[Theorem 53.3]{Sch03}}]
\label{THMbra}
In a digraph $D=(V,A)$, 
if $A$ can be partitioned into $k$ branchings, 
then 
$A$ can be partitioned into $k$ branchings 
each of which has size $\lfloor |A|/k \rfloor$ or $\lceil |A|/k \rceil$. 
\end{theorem}

The aim of this paper is to extend of Theorem \ref{THMbra} into 
two generalizations of branchings: 
\emph{matching forests} \cite{Gil82I,Gil82II,Gil82III} and \emph{$b$-branchings} \cite{KKT20}. 
An important feature is that, 
due to their structures, 
defining the equitability of matching forests and $b$-branchings is not a trivial task, 
as explained below.

\subsection{Matching Forests}
\label{SECintromf}

The concept of matching forests was introduced by Giles \cite{Gil82I,Gil82II,Gil82III}. 
A mixed graph $G=(V,E,A)$ consists of 
the set $V$ of vertices, 
the set $E$ of undirected edges, 
and 
the set $A$ of directed edges (arcs). 
We say that 
an undirected edge in $E$ \emph{covers} a vertex $v \in V$ 
if $v$ is one of the endpoints of the undirected edge, 
and 
a directed edge in $A$ \emph{covers} $v$ if $v$ is the head of the directed edge. 
A subset of edges $F \subseteq E \cup A$ is a \emph{matching forest} 
if the underlying edge set of $F$ is a forest 
and 
each vertex is covered by at most one edge in $F$. 
It is straightforward to see that 
matching forests offer a common generalization of matchings in undirected graphs 
and branchings in digraphs: 
if $F \subseteq E \cup A$ is a matching forest, 
then 
$F \cap E$ is a matching 
and 
$F \cap A$ is a branching. 

An equivalent definition of matching forests can be given in the following way. 
For a subset of undirected edges $M \subseteq E$, 
let $\partial M \subseteq V$ denote the set of vertices covered by at least one edge in $M$. 
For a subset of directed edges $B \subseteq A$, 
let $\partial B  \subseteq V$ denote the set of vertices covered by at least one arc in $B$\footnote{We believe that this notation causes no confusion on the direction of the arcs, 
since we never refer to the set of the tails of the arcs in this paper.}. 
For $F \subseteq E \cup A$, 
define $\partial F = \partial(F \cap E) \cup \partial(F \cap A)$. 
That is, $\partial F$ is the set of the vertices covered by $F$. 
Now $F \subseteq E \cup A$ is a matching forest if 
$M:=F\cap E$ is a matching, 
$B:=F\cap A$ is a branching, and 
$\partial M \cap \partial B = \emptyset$. 

Previous work on matching forests 
includes 
polynomial algorithms with polyhedral description \cite{Gil82I,Gil82II,Gil82III}, 
total dual integrality of the description \cite{Sch00MF}, 
a Vizing-type theorem on the number of matching forests partitioning the edge set $E \cup A$ \cite{Kei03}, 
and reduction to linear matroid parity \cite{Sch03}. 
More recently, 
Takazawa \cite{Tak14} showed that 
the sets of vertices covered by the matching forests form a \emph{delta-matroid} \cite{Bou87,CK88,DH86}. 
For two sets $X$ and $Y$, 
let $X \sd Y$ denote their symmetric difference, 
i.e.,\ 
$X \sd Y = (X \setminus Y) \cup (Y \setminus X)$. 
For a finite set $V$ and its subset family $\F \subseteq 2^V$, 
the set system $(V, \F)$ is a \emph{delta-matroid} if it satisfies the following exchange property: 
\begin{quote}
For each $U_1,U_2 \in \F $ and $u \in U_1 \sd U_2$, 
there exists $u' \in U_1\sd U_2$ such that $U_1 \sd \{u,u'\} \in \F$. 
\end{quote}
\begin{theorem}[\cite{Tak14}]
\label{THMdm}
For a mixed graph $G=(V,E,A)$, 
define $\F_{G} \subseteq 2^V$ by
\begin{align*}
\F_{G}=\{\partial F \mid \mbox{$F\subseteq E \cup A$ is a matching forest in $G$}\}. 
\end{align*}
Then, 
the set system $(V, \F_G)$ is a delta-matroid. 
\end{theorem}
Theorem \ref{THMdm} commonly extends matching delta-matroids \cite{Bou89,CK88} 
and matroids defined from branchings (see \cite{MT21,Tak12bibr}),
and provides some explanation of the tractability of matching forests as well as the polyhedral results mentioned above.

The most recent work on matching forests is due to Kir\'{a}ly and Yokoi \cite{KY22}, 
which discusses 
equitable partition into matching forests. 
They 
considered equitability based on the sizes of $F$, $F \cap E$, and $F \cap A$, 
and 
proved the following two theorems. 
Let $\ZZ_{++}$ denote the set of the positive integers. 
For $k \in \ZZ_{++}$, 
let $[k]$ denote the set of positive integers not greater than $k$, 
i.e., 
$[k]=\{1,\ldots, k\}$. 

\begin{theorem}[Kir\'{a}ly and Yokoi \cite{KY22}]
\label{THMky1}
Let $G=(V,E,A)$ be a mixed graph and $k\in \ZZ_{++}$. 
If $E \cup A$ can be partitioned into $k$ matching forests, 
then 
$E \cup A$ can be partitioned into $k$ matching forests $F_1,\ldots, F_k$ such that 
\begin{align}
\label{EQky1}
| |F_i| - |F_j| | \le 1, \quad 
| |M_i| - |M_j| | \le 2, \quad \mbox{and} \quad 
| |B_i| - |B_j| | \le 2
\end{align}
for every $i,j\in [k]$, 
where $M_i=F_i \cap E$ and $B_i = F_i \cap A$ for each $i \in [k]$. 
\end{theorem}

\begin{theorem}[Kir\'{a}ly and Yokoi \cite{KY22}]
\label{THMky2}
Let $G=(V,E,A)$ be a mixed graph and $k\in \ZZ_{++}$. 
If $E \cup A$ can be partitioned into $k$ matching forests, 
then 
$E \cup A$ can be partitioned into $k$ matching forests $F_1,\ldots, F_k$ such that 
\begin{align}
\label{EQky2}
| |F_i| - |F_j| | \le 2, \quad 
| |M_i| - |M_j| | \le 1, \quad \mbox{and} \quad 
| |B_i| - |B_j| | \le 2
\end{align}
for every $i,j\in [k]$, 
where $M_i=F_i \cap E$ and $B_i = F_i \cap A$ for each $i \in [k]$. 
\end{theorem}

Moreover, 
Kir\'{a}ly and Yokoi \cite{KY22} showed that 
both of Theorems \ref{THMky1} and \ref{THMky2} are best possible 
with this tricriteria equitability, 
by presenting examples in which \eqref{EQky1} and \eqref{EQky2} cannot be improved. 
That is, 
while attaining the best possible results, 
Theorems \ref{THMky1} and \ref{THMky2} mean that 
these three criteria cannot be optimized at the same time, 
which demonstrates 
a sort of obscurity of matching forests. 

In this paper, 
we introduce a new equitability, 
which builds upon a single criterion defined by 
the number of the covered vertices. 
Namely, 
our equitable-partition theorem is described as follows. 

\begin{theorem}
\label{THMmf}
Let $G=(V,E,A)$ be a mixed graph and $k \in \ZZ_{++}$. 
If $E \cup A$ can be partitioned into $k$ matching forests, 
then 
$E \cup A$ can be partitioned into $k$ matching forests $F_1,\ldots, F_k$ such that 
\begin{align}
\label{EQmf}
\left|\left| \partial F_i \right| -  \left| \partial F_j \right|\right| \le 2
\end{align}
for every $i,j\in [k]$. 
\end{theorem}

We remark that this criterion of equitability is plausible 
in the light of the delta-matroid structure of matching forests (Theorem \ref{THMdm}), 
focusing on the covered vertices rather than the edge set. 
Theorem \ref{THMmf} contrasts with Theorems \ref{THMky1} and \ref{THMky2} in that 
the value two in the right-hand side of \eqref{EQmf} is tight: 
consider the case where $G=(V,E,A)$ consists of an odd number of undirected edges forming a path 
and no directed edges, 
and $k=2$. 

We also remark that Theorem \ref{THMmf}
can indeed be derived from Lemma 4.2 in \cite{KY22}, 
as shown in \ref{APPky}. 
However, 
the proof of this lemma involves some careful case-by-case analysis, 
requiring ten types of alternating paths. 
In contrast to this, 
we present a direct and simpler proof for Theorem \ref{THMmf}, 
including only one type of alternating paths, 
by extending the argument in the proof for 
the exchangeability of matching forests by Schrijver \cite[Theorem 2]{Sch00MF}. 
Schrijver used this property to prove the total dual integrality of the linear system 
describing the matching forest polytope presented by Giles \cite{Gil82II}, 
and the delta-matroid structure of matching forests \cite{Tak14} is also derived 
from an in-depth analysis of the proof for the exchangeability. 

\subsection{$b$-branchings}

We next address equitable partition of a digraph into \emph{$b$-branchings}, 
introduced by Kakimura, Kamiyama, 
and Takazawa \cite{KKT20}. 
Let $\ZZ_{++}^V$ denote the set of the $|V|$-dimensional vectors 
of which 
each coordinate corresponds to an element in $V$
and 
each component is a positive integer. 
Let $D=(V,A)$ be a digraph and 
let $b\in \ZZ_{++}^V$. 
For $X \subseteq V$, 
we denote $b(X) = \sum_{v \in X}b(v)$. 
For 
$F \subseteq A$ and $X \subseteq V$, 
let $F[X]$ denote the set of arcs in $F$ induced by $X$. 
For 
$F \subseteq A$ and $v \in V$, 
let $d^-_{F}(v)$ denote the indegree of $v$ in the subgraph $(V,F)$, 
i.e.,\ the number of arcs in $F$ whose head is $v$. 
Now 
an arc set $B \subseteq A$ is a \emph{$b$-branching} if 
\begin{alignat}{2}
\label{EQindeg}
&{}d^-_B(v) \le b(v) \quad {}&{}{}&{}\mbox{for each $v \in V$, and} \\
\label{EQsparse}
&{}|B[X]| \le b(X) - 1 \quad {}&{}{}&{}\mbox{for each nonempty subset $X \subseteq V$}. 
\end{alignat}
Note that the branchings is a special case of $b$-branchings where $b(v)=1$ for every $v \in V$. 
That is, 
$b$-branchings provide a generalization of branchings 
in which the indegree bound of each vertex $v \in V$ can be an arbitrary positive integer $b(v)$ (Condition \eqref{EQindeg}). 
Together with Condition \eqref{EQsparse}, 
it yields a reasonable generalization of branchings admitting extensions of  
several fundamental results on branchings, 
such as a multi-phase greedy algorithm \cite{Boc71,CL65,Edm67,Ful74}, 
a packing theorem \cite{Edm73}, 
and the integer decomposition property of the corresponding polytope \cite{BT81}. 

The packing theorem on $b$-branchings leads to 
a necessary and sufficient condition for the arc set $A$ to be partitionable into $k$ $b$-branchings \cite{KKT20}. 
In this paper, 
we prove that 
an equitable partition into $k$ $b$-branchings always exists,  
provided that any partition into $k$ $b$-branchings exists. 

\begin{theorem}
\label{THMb}
Let $D=(V,A)$ be a digraph, 
$b\in \ZZ_{++}^V$, 
and $k\in \ZZ_{++}$. 
If $A$ can be partitioned into $k$ $b$-branchings, 
then 
$A$ can be partitioned into $k$ $b$-branchings 
$B_1,\ldots, B_k$ satisfying the following: 
\begin{enumerate}
\item
	\label{CONsize}
	for each $i=1,\ldots , k$, 
	the size $|B_i|$ is $\lfloor |A|/k \rfloor$ or $\lceil |A|/k \rceil$; and 
\item
	\label{CONdegree}
	for each $i=1,\ldots , k$, 
	the indegree $d^-_{B_i}(v)$ of each vertex $v \in V$ is 
	$\lfloor d^-_A(v)/k \rfloor$ or $\lceil d^-_A(v)/k \rceil$. 
\end{enumerate}
\end{theorem}

When $b(v)=1$ for every $v \in V$, 
Theorem \ref{THMb} exactly coincides with Theorem \ref{THMbra}. 
A new feature is that 
our definition of equitability of $b$-branchings is twofold: 
the number of arcs in any two $b$-branchings differ at most one (Condition \ref{CONsize}); 
and  
the indegrees of each vertex with respect to any two $b$-branchings differ at most one (Condition \ref{CONdegree}). 
Theorem \ref{THMb} means that 
the optimality of these $|V|+1$ criteria can be attained at the same time, 
which suggests some good structure of $b$-branchings. 

One consequence of Theorem \ref{THMb} is the integer decomposition property of 
the convex hull of $b$-branchings of fixed size and indegrees. 
For a polytope $P$ and a positive integer $\kappa$, 
define $\kappa P = \{x \mid \exists \mbox{$x' \in P$, $x = \kappa x'$}\}$. 
A polytope $P$ has the \emph{integer decomposition property} if, 
for every $\kappa \in \ZZ_{++}$ and every integer vector $x \in \kappa P $, 
there exist $\kappa$ integer vectors $x_1,\ldots, x_{\kappa}$ such that 
$x = x_1+\cdots+x_{\kappa}$. 

For branchings, 
Baum and Trotter \cite{BT81} showed 
that the branching polytope has 
the integer decomposition property. 
Moreover, 
McDiarmid \cite{McD83} proved the integer decomposition property of the convex hull of branchings of fixed size $\ell$. 
For $b$-branchings, 
the integer decomposition property of the $b$-branching polytope is proved in \cite{KKT20}: 
\begin{theorem}[Kakimura, Kamiyama and Takazawa \cite{KKT20}]
\label{THMidp}
Let $D=(V,A)$ be a digraph and 
$b \in \ZZ_{++}^V$. 
Then, 
the $b$-branching polytope has the integer decomposition property. 
\end{theorem}

In this paper, 
we derive the integer decomposition property of the convex hull of 
$b$-branchings of fixed size and indegrees 
from Theorems \ref{THMb} and \ref{THMidp}. 
Let $\ZZ_{+}$ denote the set of nonnegative integers, 
and 
define $\ZZ_{+}^{V'}$ for $V' \subseteq V$ in a similar way to the definition of $\ZZ_{++}^V$. 

\begin{theorem}
\label{THMidpb2}
Let $D=(V,A)$ be a digraph, 
$b \in \ZZ_{++}^V$, 
and 
$\ell\in \ZZ_+$.  
For $V' \subseteq V$, 
let 
$b' \in \ZZ_{+}^{V'}$ satisfy $b'(v) \le b(v)$ for every $v \in V'$. 
Then, 
the convex hull of incidence vectors of the $b$-branchings 
satisfying the following conditions 
has the integer decomposition property: 
\begin{enumerate}
\item 
\label{ENUsize}
the size is $\ell$; and 
\item 
\label{ENUindeg}
the indegree of each vertex $v \in V'$ is $b'(v)$. 
\end{enumerate}
\end{theorem}

By taking $V'= \emptyset$ in Theorem \ref{THMidpb2}, 
we obtain the following, 
which is an extension of the result of McDiarmid \cite{McD83}.

\begin{theorem}
\label{THMidpbsize}
Let $D=(V,A)$ be a digraph, 
$b \in \ZZ_{++}^V$, 
and 
$\ell\in \ZZ_+$. 
Then, 
the convex hull of the incidence vectors of the $b$-branchings of size $\ell$ has the integer decomposition property. 
\end{theorem}

We can also 
remove Condition \ref{ENUsize} in Theorem \ref{THMidpb2}. 
\begin{theorem}
\label{THMidpindeg}
Let $D=(V,A)$ be a digraph and 
$b \in \ZZ_{++}^V$. 
For $V' \subseteq V$, 
let 
$b' \in \ZZ_{+}^{V'}$ satisfy $b'(v) \le b(v)$ for every $v \in V'$. 
Then, 
the convex hull of the incidence vectors of the $b$-branchings 
for which the indegree of each vertex $v \in V'$ is $b'(v)$ 
has the integer decomposition property. 
\end{theorem}

\subsection{Organization of the Paper}

The remainder of the paper is organized as follows. 
Section \ref{SECmf} is devoted to a proof for Theorem \ref{THMmf} on equitable partition into matching forests. 
In Section \ref{SECb}, 
we prove Theorem \ref{THMb} on equitable partition into $b$-branchings, 
and then derive Theorems \ref{THMidpb2}--\ref{THMidpindeg} on the integer decomposition property of the related polytopes.

\section{Equitable Partition into Matching Forests}
\label{SECmf}

The aim of this section is to prove Theorem \ref{THMmf}. 
Let $G=(V,E,A)$ be a mixed graph. 
For a branching $B \subseteq A$, 
let $R(B) = V \setminus \partial B$, 
which represents the set of root vertices of $B$. 
Similarly, 
for a matching $M \subseteq E$, 
define $R(M)= V \setminus \partial M$. 
Note that, 
for a matching $M$ and a branching $B$, 
their union $M \cup B$ is a matching forest if and only if $R(M) \cup R(B) = V$. 

A \emph{source component} $X$ in a digraph $D=(V,A)$
is a strong component in $D$ such that no arc in $A$ enters $X$. 
In what follows, 
a source component is often denoted by its vertex set. 
Observe that, 
for a vertex subset $V' \subseteq V$, 
there exists a branching $B$ satisfying $R(B)= V'$ 
if and only if $|V' \cap X| \ge 1$ for every source component $X$ in $D$. 
This fact is extended to the following lemma on the partition of the arc set into two branchings, 
which can be derived from Edmonds' disjoint branchings theorem \cite{Edm73}. 

\begin{lemma}[Schrijver \cite{Sch00MF}]
\label{LEMsch}
Let $D=(V,A)$ be a digraph, 
and 
$B_1$ and $B_2$ are branchings in $D$ partitioning $A$. 
Then, 
for two vertex sets 
$R_1',R_2' \subseteq V$ such that 
$R_1' \cup R_2' = R(B_1) \cup R(B_2)$ 
and 
$R_1' \cap R_2' = R(B_1) \cap R(B_2)$, 
the arc set $A$ 
can be partitioned into two branchings $B_1'$ and $B_2'$ 
such that 
$R(B_1')=R_1'$ and $R(B_2')=R_2'$ 
if and only if 
$$
|R_1' \cap X| \ge 1 \quad \mbox{and} \quad |R_2' \cap X| \ge 1 \quad \mbox{for each source component $X$ in $D$}. 
$$
\end{lemma}

We remark that Schrijver \cite{Sch03} derived Theorem \ref{THMbra} from Lemma \ref{LEMsch}. 
Here we prove that Lemma \ref{LEMsch} further leads to Theorem \ref{THMmf}. 

\begin{proof}[of Theorem \ref{THMmf}]
The case $k=1$ is trivial, 
and thus let $k \ge 2$. 
Let $F_1,\ldots ,F_k$ be matching forests minimizing
\begin{align}
\label{EQdiffsum}
\sum_{1 \le i < j \le k}
||\partial F_i| - |\partial F_j||
\end{align}
among those partitioning $E \cup A$. 
We prove that 
every pair of 
$F_i$ and $F_j$ ($i,j \in [k]$) attains \eqref{EQmf}. 

Suppose to the contrary that 
\eqref{EQmf} does not hold for some $i,j \in [k]$. 
Without loss of generality, 
assume 
\begin{align}
\label{EQdiff3}
 |\partial F_1| - |\partial F_2|  \ge 3. 
\end{align}
Let $A' = B_1 \cup B_2$. 
Denote the family of source components in $(V,A')$ by $\X'$. 
If a vertex $v \in V$ belongs to $R(B_1) \cap R(B_2)$, 
then $v$ has no incoming arc in $A'$, 
and hence 
$v$ itself forms a source component in $(V,A')$. 
Thus, 
for $X \in \X'$ with $|X| \ge 2$, 
it follows that 
$X \cap R(B_1)$ and $X \cap R(B_2)$ are not empty and disjoint with each other. 
Denote the family of such $X \in \X'$ by $\X''$, 
i.e.,\ 
$\X'' = \{ X \in \X' \mid |X|\ge 2 \}$. 
For each $X\in \X''$, 
take 
a pair $e_X$ of vertices of which one vertex is in $R(B_1)$ 
and the other in $R(B_2)$.   
Denote $N = \{ e_X \mid X \in \X'' \}$. 
Note that $N$ is a matching. 

Consider an undirected graph $H=(V, M_1 \cup M_2 \cup N)$. 
Observe that each vertex $v \in V$ has degree at most two: 
if a vertex $v \in V$ is covered by both $M_1$ and $M_2$, 
then 
it follows that 
$v \in R(B_1) \cap R(B_2)$, 
implying that $v$ is not covered by $N$. 
Thus, 
$H$ consists of a disjoint collection of paths, 
some of which are 
possibly isolated vertices, and cycles. 

\begin{proposition}
\label{PROPinternal}
For a vertex $v$ in $H$ with degree exactly two, 
it holds that $v \in \partial F_1 \cap \partial F_2$. 
\end{proposition}

\begin{proof}[of Proposition \ref{PROPinternal}]
It is clear if $v \in \partial M_1 \cap \partial M_2$, 
and suppose not. 
Without loss of generality, 
assume $v \in \partial M_1 \cap \partial N$. 
It then follows from $v \in \partial M_1$ that 
$v \in R(B_1)$. 
Since $v \in \partial N$, 
this implies that $v \not \in R(B_2)$. 
We thus conclude $v \in \partial B_2 \subseteq \partial F_2$.
\end{proof}

By Proposition \ref{PROPinternal}, 
each vertex $u \in \partial F_1 \sd \partial F_2$ is an endpoint of a path in $H$. 
It then follows from \eqref{EQdiff3} that 
there must exist a path $P$ such that 
\begin{itemize}
\item
one endpoint $u$ of $P$ belongs to $\partial F_1 \setminus \partial F_2$, and 
\item
the other endpoint $u'$ of $P$ belongs to 
$\partial F_1$.
\end{itemize}
We remark that 
$u'$ may or may not belong to $\partial F_2$. 
It may also be the case that $u'$ is null, i.e.,\ 
$u$ is an isolated vertex which by itself forms $P$. 

Denote the set of vertices in $P$ by $V(P)$, 
and 
the set of edges in $P$ belonging to $M_1 \cup M_2$ by $E(P)$. 
Define 
$M_1' = M_1 \sd E(P)$ 
and 
$M_2' = M_2 \sd E(P)$. 
It follows that 
$M_1'$ and $M_2'$ are matchings satisfying 
\begin{align*}
{}&{}R( M_1') = (R(M_1) \setminus V(P)) \cup  (R(M_2) \cap V(P)), \\
{}&{}R(M_2') = (R(M_2) \setminus V(P)) \cup  (R(M_1) \cap V(P) ).
\end{align*}
Also define 
\begin{align*}
{}&{}R_1' = (R(B_1) \setminus V(P)) \cup (R(B_2) \cap V(P)), \\
{}&{}R_2' = (R(B_2) \setminus V(P)) \cup (R(B_1) \cap V(P)). 
\end{align*}
It then follows that 
$R_1' \cup R_2' = R(B_1) \cup R(B_2)$ 
and 
$R_1' \cap R_2' = R(B_1) \cap R(B_2)$. 
It also follows from the construction of $H$ that 
$R_i'  \cap X\neq \emptyset$ for every $X \in \X'$ and for $i=1,2$. 
Thus, 
by Lemma \ref{LEMsch}, 
the arc set $A'$ can be partitioned into two branchings $B_1'$ and $B_2'$ such that 
$R(B_1') = R_1'$ and $R(B_2') = R_2'$. 
Then it holds that 
\begin{align*}
R(M_1') \cup R(B_1')
{}&{}= R(M_1') \cup R_1' \\
{}&{}= ((R(M_1) \cup R(B_1)) \setminus V(P)) \cup ((R(M_2) \cup R(B_2)) \cap V(P)) \\
{}&{}= (V \setminus V(P)) \cup (V \cap V(P)) \\
{}&{}=V,
\end{align*}
and hence $F_1':=M_1' \cup B_1'$ is a matching forest in $G$. 
This is also the case with 
$F_2':=M_2' \cup B_2'$. 

Now we have two disjoint matching forests $F_1'$ and $F_2'$ 
such that 
$F_1' \cup F_2' = F_1 \cup F_2$. 
Moreover, 
by the definition of $P$, 
we have that 
$$||\partial F_1'| - |\partial F_2'|| = ||\partial F_1| - |\partial F_2|| - 2
\quad \mbox{or} \quad 
||\partial F_1'| - |\partial F_2'|| = ||\partial F_1| - |\partial F_2|| - 4,$$ 
and in particular, 
by \eqref{EQdiff3},  
$$||\partial F_1'| - |\partial F_2'|| < ||\partial F_1| - |\partial F_2||. $$ 
It also follows that 
\begin{align*}
{}&{}\sum_{i \in [k] \setminus \{1,2\}} \left( \left| \left|\partial F_1'\right| - \left|\partial F_i\right| \right| + \left|\left|\partial F_2'\right| - \left|\partial F_i\right|\right|\right) \\
\le 
{}&{}\sum_{i \in [k] \setminus \{1,2\}} \left( \left| \left|\partial F_1\right| - \left|\partial F_i\right| \right| + \left|\left|\partial F_2\right| - \left|\partial F_i\right|\right|\right) .
\end{align*}
This contradicts the fact that $F_1,\ldots, F_k$ minimize \eqref{EQdiffsum}, 
and thus completes the proof of the theorem. 
\end{proof}

We remark that, 
if an arbitrary partition of $E \cup A$ into $k$ matching forests is given, 
a partition of $E \cup A$ into $k$ matching forests satisfying \eqref{EQmf} can be found in polynomial time. 
This can be done 
by repeatedly applying the update of two matching forests described in the above proof. 
The time complexity follows from the fact that  
each update can be done in polynomial time and 
decreases the value \eqref{EQdiffsum} by at least two.

\section{Equitable Partition into $b$-branchings}
\label{SECb}

In this section 
we first prove Theorem \ref{THMb}, 
and then derive Theorems \ref{THMidpb2}--\ref{THMidpindeg}. 
In proving Theorem \ref{THMb}, 
we make use of the following lemma, 
which is an extension of Lemma \ref{LEMsch} to $b$-branchings. 


\begin{lemma}[\cite{Tak22}]
\label{LEMbpartition}
Let $D=(V,A)$ be a digraph and $b \in \ZZ_{++}^V$. 
Suppose that  
$A$ can be partitioned into two $b$-branchings $B_1, B_2 \subseteq A$. 
Then, 
for two vectors $b_1', b_2' \in \ZZ_{++}^V$ satisfying $b_1'\leq b$, $b_2'\leq b$, and $b_1' + b_2'=d_A^-$, 
the arc set $A$ can be partitioned into two $b$-branchings $B_1'$ and $B_2'$ such that 
$d^-_{B_1'}=b_1'$ and $d^-_{B_2'}=b_2'$ if and only if 
\begin{align*}
b_1'(X) < b(X) \quad \mbox{and} \quad b_2'(X) < b(X) \quad \mbox{for each source component $X$ in $D$}. 
\end{align*}
\end{lemma}

We now prove Theorem \ref{THMb}.

\begin{proof}[of Theorem \ref{THMb}]
The case $k=1$ is trivial, 
and thus let $k \ge 2$. 
Let $B_1,\ldots, B_k$ be $k$ $b$-branchings 
minimizing 
\begin{multline}
\label{EQdiff}
\sum_{i\in [k]} 
\left(
	\min \left\{  \left| |B_i| - \left\lfloor\dfrac{|A|}{k} \right\rfloor  \right|, \left| |B_i| - \left \lceil \dfrac{|A|}{k} \right \rceil  \right|  \right\} 
	\right.\\ \left.
	+ \sum_{v \in V} \min \left\{  \left| |d_{B_i}^-(v)| - \left\lfloor\dfrac{d_A^-(v)}{k} \right\rfloor  \right|, \left| |d_{B_i}^-(v)| - \left\lceil\dfrac{d_A^-(v)}{k} \right\rceil  \right|  \right\} 
\right)
\end{multline}
among those partitioning $A$. 

Suppose to the contrary that 
Condition \ref{CONsize} or \ref{CONdegree} does not hold for some $i\in [k]$. 
Then, 
there exists $j \in [k]$ such that 
\begin{align}
\label{EQij1}
\min\left\{|B_i|,|B_j|\right\} < \dfrac{|A|}{k} < \max\left\{|B_i|,|B_j|\right\}, \ 
\left| |B_i| - |B_j| \right| \ge 2, 
\end{align}
or 
there exist $j \in [k]$ and  $v \in V$ such that 
\begin{align}
\label{EQij2}
\min \left\{d_{B_i}^-(v) , d_{B_j}^-(v) \right\} < \frac{d_A^-(v)}{k} < \max \left \{d_{B_i}^-(v) , d_{B_j}^-(v) \right\}, \ 
\left| d_{B_i}^-(v) - d_{B_j}^-(v) \right| \ge 2. 
\end{align}
Without loss of generality, 
let $i=1$ and $j=2$, 
and 
denote $b_1=d_{B_1}^-$ and $b_2=d_{B_2}^-$. 
Let $D' = (V, B_1 \cup B_2)$. 
Since $B_1$ and $B_2$ are $b$-branchings, 
it directly follows the definition of $b$-branchings that 
\begin{alignat}{2}
\label{EQ1-4}
&{}b_1(v) \le b(v)& \quad {}&{}\mbox{for each $v\in V$}, \\
&{}b_2(v) \le b(v)& \quad {}&{}\mbox{for each $v\in V$}, \\
&{}b_1(X) \le b(X)-1&   \quad {}&{}\mbox{for each source component $X$ in $D'$},\\
\label{EQ4-4}
&{}b_2(X) \le b(X)-1&   \quad {}&{}\mbox{for each source component $X$ in $D'$}. 
\end{alignat}

Let $\mathcal{X}$ be the set of source components $X$ in $D'$ 
such that $b_1(X) + b_2(X)$ is even, 
and 
let $\mathcal{Y}$ be the set of source components $Y$ in $D'$ 
such that $b_1(Y) + b_2(Y)$ is odd. 
Then, 
define $b_1',b_2'\in \ZZ_+^V$ 
satisfying 
$b_1' + b_2'=  b_1+b_2$ in the following manner. 
\begin{itemize}
\item
For all $X \in \mathcal{X}$, 
take $b_1'(v),b_2'(v) \in \ZZ_+$ for all $v \in X$ so that 
	\begin{alignat*}{2}
	&{}b_1'(v) = b_2'(v) = \frac{b_1(v) + b_2(v)}{2}{}&{} \quad {}&{}\mbox{if $b_1(v) + b_2(v)$ is even;} \\
	&{}
	|b_1'(v) - b_2'(v)| = 1 {}&{}\quad {}&{}\mbox{if $b_1(v) + b_2(v)$ is odd; and} \\ 
	&{}
	b_1'(X) =b_2'(X). {}&{}{}&{}
	\end{alignat*}
\item
For all $Y \in \mathcal{Y}$, 
take $b_1'(v),b_2'(v) \in \ZZ_+$ for all $v \in Y$ so that 
	\begin{alignat*}{2}
	{}&{}
	b_1'(v) = b_2'(v) = \frac{b_1(v) + b_2(v)}{2} {}&{} \quad {}&{}\mbox{if $b_1(v) + b_2(v)$ is even;} \\
	{}&{}
	|b_1'(v) - b_2'(v)| = 1 {}&{} \quad {}&{} \mbox{if $b_1(v) + b_2(v)$ is odd;} \\
	{}&{}
	\left|  b_1'(Y) - b_2'(Y) \right| = 1 {}&{}\quad {}&{}\mbox{for every $Y \in \Y$; and} \\
	{}&{} 
	\left|  \sum_{Y \in \Y}b_1'(Y) - \sum_{Y \in \Y}b_2'(Y) \right| \le 1.{}&{}{}&{} 
	\end{alignat*}
\item
For $v \in V\setminus (\bigcup_{X \in \X \cup \Y}X)$, 
take $b_1'(v),b_2'(v) \in \ZZ_+$ so that 
\begin{align*}
&{}|b_1'(v) - b_2'(v)| \le 1 \quad \mbox{for every $\displaystyle v \in V\setminus \bigcup_{X \in \X \cup \Y}X$; and} \\
&{}|b_1'(V) - b_2'(V)| \le 1. 
\end{align*}
\end{itemize}
Now 
it directly follows from \eqref{EQ1-4}--\eqref{EQ4-4} that 
$b_1' \le b$, 
$b_2' \le b$, 
and 
\begin{align*}
 b_1'(X) \le b(X)-1, \quad  b_2'(X) \le b(X)-1 \quad (X \in \X \cup \Y).
\end{align*}
It then follows from Lemma \ref{LEMbpartition} that 
there exist $b$-branchings $B_1'$ and $B_2'$ such that $B_1' \cup B_2' =  B_1 \cup B_2$, 
$d_{B_1}' = b_1'$, 
and 
$d_{B_2}' = b_2'$. 
For these two $b$-branchings $B_1'$ and $B_2'$, 
we have that 
\begin{align}
\label{EQimprv}
&{}\left| |B'_1| - |B'_2| \right| \le 1 
\quad\mbox{and}\quad
\left| d_{B'_1}^-(v) - d_{B'_2}^-(v) \right| \le 1 \quad \mbox{for every $v \in V$}.  
\end{align}
Therefore, 
we can 
strictly decrease the value \eqref{EQdiff} 
by replacing $B_1$ and $B_2$, 
which satisfy \eqref{EQij1} or \eqref{EQij2}, with $B_1'$ and $B_2'$, 
which satisfy \eqref{EQimprv}. 
This contradicts the fact that $B_1, \ldots, B_k$ minimize \eqref{EQdiff}. 
\end{proof}

We remark that 
a partition of $A$ into $k$ $b$-branchings satisfying Conditions \ref{CONsize} and \ref{CONdegree} in Theorem \ref{THMb} can be found in polynomial time. 
First, 
we can check if there exists a partition of $A$ into $k$ $b$-branchings 
and find one if exists in polynomial time \cite{KKT20}. 
If this partition does not satisfy Conditions \ref{CONsize} and \ref{CONdegree}, 
then we repeatedly apply the update of two $b$-branchings as shown in the above proof, 
which 
can be done in polynomial time and 
strictly decreases the value \eqref{EQdiff}. 




We conclude this paper by deriving Theorems \ref{THMidpb2}--\ref{THMidpindeg} from Theorem \ref{THMb}. 
Here we only present a proof for Theorem \ref{THMidpb2}: 
Theorem \ref{THMidpbsize} is a special case $X = \emptyset$ of Theorem \ref{THMidpb2}; 
and 
Theorem \ref{THMidpindeg} can be proved by the same argument. 


\begin{proof}[of Theorem \ref{THMidpb2}]
Denote the convex hull of the $b$-branchings in $D$ by $P$, 
and that of the $b$-branchings in $D$  
satisfying \ref{ENUsize} and \ref{ENUindeg} 
by $Q$. 
Take a positive integer $\kappa \in \ZZ_{++}$, 
and let $x$ be an integer vector in $\kappa Q$. 
For a vertex $v\in V$, 
let $\delta^- (v) \subseteq A$ denote the set of arcs whose head is $v$. 
Observe that
\begin{align}
\label{EQsize1}
{}&{}x(A) = \kappa \cdot \ell, \\
\label{EQindeg1}
{}&{}x(\delta^- (v)) = \kappa \cdot b'(v) \quad \mbox{for each $v \in V'$}. 
\end{align}
follow from the definition of $Q$. 

Let $D'=(V,A')$ be a digraph 
obtained from $D$ by replacing each arc $a\in A$ by 
$x_a$ parallel arcs. 
Then, 
since $x \in \kappa Q \subseteq \kappa P$, 
it follows from the integer decomposition property of the $b$-branching polytope (Theorem \ref{THMidp}) 
that 
$x$ is the sum of the incidence vectors of $\kappa $ $b$-branchings, 
i.e., 
$A'$ 
can be partitioned into $\kappa$ $b$-branchings. 
Here it directly follows from \eqref{EQsize1} and \eqref{EQindeg1} that 
\begin{align*}
{}&{}|A'| = |x(A)| = \kappa \cdot \ell, \\
{}&{}|d_{A'}^- (v)| =  \kappa \cdot b'(v) \quad (v \in V'),
\end{align*}
and thus 
it follows from Theorem \ref{THMb} that 
$A'$ can be partitioned into $\kappa$ $b$-branchings $B_1,\ldots, B_{\kappa}$ 
such that 
\begin{alignat*}{2}
{}&{}|B_i| = \ell{}&{} \quad {}&{}(i \in [\kappa]), \\
{}&{}d_{B_i}^-(v)=b'(v){}&{} \quad{}&{}(i \in [\kappa], v \in V'). 
\end{alignat*}
Therefore, 
$x$ can be represented as the sum of the incidence vectors of $\kappa$ $b$-branchings satisfying \ref{ENUsize} and \ref{ENUindeg}, 
i.e.,\ 
integer vectors in $Q$, 
which completes the proof. 
\end{proof}

\section*{Acknowledgements}
The author thanks an anonymous reviewer for helpful comments on the relation to \cite{KY22}.

\appendix 
\def\thesection{Appendix \Alph{section}}

\section{Alternative Proof for Theorem \ref{THMmf}}
\label{APPky}

Here we present an alternative proof for Theorem \ref{THMmf}. 
On the way to proving Theorem \ref{THMky1}, 
Kir\'{a}ly and Yokoi \cite{KY22} showed the following lemma. 
\begin{lemma}[Lemma 4.2 in \cite{KY22}]
\label{LEMky}
Let $G=(V,E,A)$ be a mixed graph. 
If $E\cup A$ can be partitioned into two matching forests, 
then 
$E\cup A$ can be partitioned into two matching forests $F_1$ and $F_2$ such that
$\left| |F_1| - |F_2| \right| \leq 1$ and 
$ | |F_1| - |F_2| | + | |M_1| - |M_2| |  \leq 2$, 
where 
$M_1=F_1 \cap E$ and $M_2 = F_2 \cap E$. 
\end{lemma}
Theorem \ref{THMmf} can be derived from Lemma \ref{LEMky} in the following way. 

\begin{proof}[of Theorem \ref{THMmf}]
We can assume $k \ge 2$. 
Let $F_1,\ldots ,F_k$ be matching forests minimizing
\begin{align}
\label{EQdiffsumapp}
\sum_{1 \le i < j \le k}
\left||\partial F_i| - |\partial F_j|\right|
\end{align}
among those partitioning $E \cup A$. 

Suppose to the contrary that 
\eqref{EQmf} does not hold for some $i,j \in [k]$. 
Without loss of generality, 
assume 
\begin{align}
\label{EQdiffapp}
 ||\partial F_1| - |\partial F_2| | \ge 3. 
\end{align}
Then, 
it follows from Lemma \ref{LEMky} that $F_1 \cup F_2$ can be partitioned into 
two matching forests 
$F_1'$ and $F_2'$ such that
\begin{align*}
&| |F_1'| - |F_2'| | \leq 1, \\
& | |F_1'| - |F_2'| | + | |M_1'| - |M_2'| |  \leq 2,
\end{align*}
where 
$M_1'=F_1' \cap E$ and $M_2' = F_2' \cap E$. 
Observe that 
$|\partial F|=|F| + |F \cap E|$ holds 
for an arbitrary matching forest $F$. 
Thus, 
\begin{align*}
| |\partial F_1'| - |\partial F_2'| |
{}&{}=     | (|F_1'| + |M_1'|) - (|F_2'| + |M_2'|) |  \\
{}&{}\leq  | |F_1'| - |F_2'|| + ||M_1'| - |M_2'| | \\
{}&{}\leq 2.
\end{align*}
It then follows from \eqref{EQdiffapp} 
that 
$| |\partial F_1'| - |\partial F_2'| | < ||\partial F_1| - |\partial F_2| |$.
It also follows that 
\begin{align*}
{}&{}\sum_{i \in [k] \setminus \{1,2\}} \left( \left| \left|\partial F_1'\right| - \left|\partial F_i\right| \right| + \left|\left|\partial F_2'\right| - \left|\partial F_i\right|\right|\right) \\
\le 
{}&{}\sum_{i \in [k] \setminus \{1,2\}} \left( \left| \left|\partial F_1\right| - \left|\partial F_i\right| \right| + \left|\left|\partial F_2\right| - \left|\partial F_i\right|\right|\right) .
\end{align*}
This contradicts the fact that $F_1,\ldots, F_k$ minimize \eqref{EQdiffsumapp}, 
and thus completes the proof of the theorem. 
\end{proof}



\bibliographystyle{myjorsj2}
\bibliography{../../../../../refs}

\end{document}